\documentclass{amsart}

\usepackage{amssymb}
\usepackage[hidelinks]{hyperref}
\usepackage{todonotes}

\theoremstyle{plain}
\newtheorem{theorem}{Theorem}[section]
\newtheorem{lemma}[theorem]{Lemma}

\newtheorem{proposition}[theorem]{Proposition}

\theoremstyle{definition}
\newtheorem{definition}[theorem]{Definition}

\newtheorem{question}{Question}

\theoremstyle{remark}

\newcommand{\cI}{\mathcal{I}}
\newcommand{\I}{\cI}
\newcommand{\cJ}{\mathcal{J}}
\newcommand{\J}{\cJ}

\newcommand{\cP}{\mathcal{P}}

\newcommand{\cS}{\mathcal{S}}

\newcommand{\fin}{\mathrm{Fin}}
\newcommand{\Exh}{\mathrm{Exh}}

\newcommand{\conv}{\mathrm{conv}}

\begin{document}


\title{On extendability to $F_\sigma$ ideals}


\author[A.~Kwela]{Adam Kwela}
\address[Adam Kwela]{Institute of Mathematics\\ Faculty of Mathematics\\ Physics and Informatics\\ University of Gda\'{n}sk\\ ul.~Wita  Stwosza 57\\ 80-308 Gda\'{n}sk\\ Poland}
\email{Adam.Kwela@ug.edu.pl}
\urladdr{http://kwela.strony.ug.edu.pl/}






\begin{abstract}
Answering in negative a question of M. Hru\v{s}\'ak, we construct a Borel ideal not extendable to any $F_\sigma$ ideal and such that it is not Kat\v{e}tov above the ideal $\mathrm{conv}$. 
\end{abstract}


\maketitle


\section{Introduction}

We use standard set-theoretic notation. In particular, a collection $\mathcal{I}$ of subsets of a set $X$ is called an ideal on $X$ if it is closed under subsets and finite unions of its elements. We assume additionally that $\mathcal{P}(X)$ (i.e., the power set of $X$) is not an ideal, and that every ideal contains all finite subsets of $X$ (hence, $X=\bigcup\I$). All ideals considered in this paper are defined on infinite countable sets. 

We treat the power set $\mathcal{P}(X)$ as the space $2^X$ of all functions $f:X\rightarrow 2$ (equipped with the product topology, where each space $2=\left\{0,1\right\}$ carries the discrete topology) by identifying subsets of $X$ with their characteristic functions. Thus, we can talk about descriptive complexity of subsets of $\mathcal{P}(X)$ (in particular, of ideals on $X$). 

This article is motivated by a problem of M. Hru\v{s}\'ak concerning characterization of Borel ideals that can be extended to a $\bf{\Sigma^0_2}$ ideal (i.e., such Borel ideals $\I$ that there is a $\bf{\Sigma^0_2}$ ideal $\J$ with $\I\subseteq\J$). In order to formulate this question in a precise way, we need to recall two definitions:
\begin{itemize}
\item if $\I$ and $\J$ are ideals then we say that $\I$ is below $\J$ in the Kat\v{e}tov preorder (and write $\I\leq_K\J$), if there is $f:\bigcup\J\to \bigcup\I$ such that $f^{-1}[A]\in\J$ for each $A\in\I$ (cf. \cite[Subsection 1.3]{Hrusak} or \cite[Subsection 1.5]{Meza});
\item by $\conv$ we denote the ideal on $\mathbb{Q}\cap[0,1]$ generated by sequences in $\mathbb{Q}\cap[0,1]$ that are convergent in $[0,1]$, i.e., $A\subseteq\mathbb{Q}\cap[0,1]$ belongs to the ideal $\conv$ if it can be covered by finitely many such sequences (cf. \cite[Subsection 3.4]{Hrusak} or \cite[Subsection 1.6]{Meza}).
\end{itemize}
A property of ideals can often be expressed by finding a critical ideal (in sense of the Kat\v{e}tov preorder) with respect to this property. This approach proved to be especially effective in many papers including \cite{Hrusak}, \cite{Hrusak2}, \cite{WR}, \cite{EUvsSD} and \cite{Meza}. The above mentioned question of M. Hru\v{s}\'ak is the following: 

\begin{question}
Is it true that, if $\I$ is a Borel ideal then either $\conv\leq_K\I$ or there is a $\bf{\Sigma^0_2}$ ideal containing $\I$?
\end{question}

This problem has been asked in \cite[Question 5.16]{Hrusak} and repeated in \cite[Question 5.8]{Hrusak2}. It is known that $\conv$ is a $\bf{\Sigma^0_4}$ ideal that cannot be extended to any $\bf{\Sigma^0_2}$ ideal (by \cite[Propositions 3.4 and 4.1]{Gdansk} and \cite[Subsection 2.7]{Meza}). Thus, a positive answer to M. Hru\v{s}\'ak's question would establish a combinatorial characterization of Borel ideals extendable to $\bf{\Sigma^0_2}$ ideals. However, in this paper we answer it in the negative.

A very similar problem to the above one has been posed by D. Meza-Alc\'antara in \cite[Question 4.4.6]{Meza}: Is it true that, if $\I$ is a Borel ideal then either $\conv\leq_K\I|A$ (here $\I|A=\{B\subseteq A:\ B\in\I\}$) for some $A\notin\I$ or there is a $\bf{\Sigma^0_2}$ ideal containing $\I$? This question remains open.

$\bf{\Sigma^0_2}$ ideals are closely related to the notion of P$^+$-ideals. We say that an ideal $\I$ on $X$ is a P$^+$-ideal if for each decreasing sequence $(A_n)$ of sets not belonging to $\I$ one can find $B\notin\J$ such that $B\setminus A_n$ is finite for all $n\in\omega$ (\cite[Definition 2.2.3]{Meza}). This notion has been studied by for instance in \cite[Subsection 1.1]{Hrusak}. By \cite[Theorem 3.2.7]{Meza}, a Borel ideal is extendable to a $\bf{\Sigma^0_2}$ ideal if and only if it is extendable to a P$^+$-ideal. Therefore, original question of M. Hru\v{s}\'ak can be reformulated in the following way: Is is true that a Borel ideal $\I$ is extendable to a P$^+$-ideal if and only if $\conv\not\leq_K\I$?

It is worth mentioning that M. Hru\v{s}\'ak's conjecture holds for all analytic P-ideals (an ideal $\I$ is called a P-ideal if for every $(A_n)\subseteq\I$ there is $A\in\I$ with $A\setminus A_n$ finite for all $n$), i.e., an analytic P-ideal $\I$ is extendable to a $\bf{\Sigma^0_2}$ ideal if and only if $\conv\not\leq_K\I$ (cf. \cite[Theorem 4.2]{Gdansk} and \cite[Subsection 2.7]{Meza}). Moreover, by \cite[Theorem 3.2.14]{Meza}, if $\I$ is a Borel ideal such that the forcing $\cP(\omega)/\I$ is proper then either it is extendable to a $\bf{\Sigma^0_2}$ ideal or there is $A\notin\I$ with $\conv\leq_K\I|A$.

In Section 2 we introduce some necessary notions. Section 3 contains the solution of M. Hru\v{s}\'ak's question. Section 4 is devoted to some concluding remarks.

\section{Preliminaries}

If $\I$ and $\J$ are ideals on $X$ and $Y$, respectively, then the ideal:
$$\I\oplus\J=\left\{A\subseteq (\{0\}\times X)\cup(\{1\}\times Y):\right.$$
$$\left.\{x\in X:\ (0,x)\in A\}\in\I\text{ and }\{y\in Y:\ (1,y)\in A\}\in\J\right\}$$
is their disjoint sum (see \cite{Farah}). Two ideals $\I$ and $\J$ are isomorphic if there is a bijection $f:\bigcup\J\to \bigcup\I$ such that 
$$f^{-1}[A]\in\J\ \Longleftrightarrow\ A\in\I$$
for all $A\subseteq\bigcup\I$. Isomorphisms of ideals have been deeply studied for instance in \cite{Tryba} (see also \cite{Farah} and \cite{EUvsSD}). 

Recall the definition of the classical ideal of asymptotic density zero sets: 
$$\I_{d}=\left\{A\subseteq\omega:\ \lim_{n\to\infty}\frac{|A\cap[0,n]|}{n+1}=0\right\},$$
where by $[0,n]$ we denote the set $\{0,1,\ldots,n\}$. This ideal has been deeply investigated in the past in the context of convergence (see e.g. \cite{Fast}, \cite{Fridy}, \cite{Salat} and \cite{Steinhaus}) as well as from the set-theoretic point of view (see e.g. \cite{Farah}, \cite{Just} and \cite{EUvsSD}).

In our considerations we will need some new notions which we introduce below. 

In this paper we denote by $\cS$ the family of all sequences in $(0,\frac{1}{2}]$ decreasing to zero, i.e.:
$$\cS=\left\{(\alpha_n)\in\left(0,\tfrac{1}{2}\right]^\omega:\ \lim_n\alpha_n=0\text{ and }\alpha_{n+1}<\alpha_n\text{ for all }n\in\omega\right\}.$$
For $x\in[0,1]$ and $r>0$ by $B(x,r)$ we denote the ball of radius $r$ centered at $x$, i.e., $B(x,r)=(x-r,x+r)$. Note that for every $(\alpha_n)\in\cS$ and every $x\in[0,1]$ we have $(B(x,\alpha_0)\setminus B(x,\alpha_1))\cap[0,1]\neq\emptyset$.

\begin{definition}
If $(\alpha_n)\in\cS$ and $\I$ is an ideal on $\omega$ then we say that a sequence $(x_k)\in[0,1]^\omega$ converges $\I$-quickly with respect to $(\alpha_n)$ if there is $x\in[0,1]$ such that $\lim_k x_k=x$ and:
$$\left\{n\in\omega:\ \{x_k:\ k\in\omega\}\cap (B(x,\alpha_n)\setminus B(x,\alpha_{n+1}))\neq\emptyset\right\}\in\I.$$
\end{definition}

\begin{definition}
If $(\alpha_n)\in\cS$ and $\I$ is an ideal on $\omega$ then $\conv(\I,(\alpha_n))$ is the ideal on $[0,1]\cap\mathbb{Q}$ generated by all sequences converging $\I$-quickly with respect to $(\alpha_n)$, i.e., $A\subseteq\mathbb{Q}\cap[0,1]$ belongs to $\conv(\I,(\alpha_n))$ if it can be covered by finitely many sequences converging $\I$-quickly with respect to $(\alpha_n)$.
\end{definition}

Our counterexample will be of the form $\conv(\I,(\alpha_n))$. We start with a simple general result concerning such ideals.

\begin{proposition}
\label{analytic}
If $\I$ is an analytic ideal on $\omega$ and $(\alpha_n)\in\cS$ then $\conv(\I,(\alpha_n))$ is also analytic.
\end{proposition}

\begin{proof}
Fix an ideal $\I$ on $\omega$ and $(\alpha_n)\in\cS$. Observe that:
$$\conv(\I,(\alpha_n))=\left\{A\subseteq[0,1]\cap\mathbb{Q}:\ \exists_{k\in\omega}\ \exists_{x_0,\ldots,x_{k-1}\in[0,1]}\ \left(\forall_{i<k}\ A\in f_{x_i}^{-1}[\I]\right)\ \wedge\right.$$
$$\left. \left(\forall_{m\in\omega}\ \exists_{F\in[\mathbb{Q}\cap[0,1]]^{<\omega}}\ A\subseteq F\cup\bigcup_{i<k}B\left(x_i,\tfrac{1}{m+1}\right)\right)\ \right\},$$
where $f_x:\mathcal{P}([0,1]\cap\mathbb{Q})\to\mathcal{P}(\omega)$, for each $x\in[0,1]$, is given by:
$$f_{x}(A)=\left\{n\in\omega:\ A\cap \left(B(x,\alpha_n)\setminus B(x,\alpha_{n+1})\right)\neq\emptyset\right\},$$
for all $A\subseteq[0,1]\cap\mathbb{Q}$.

Notice that if $x\in[0,1]$ and $l\in\omega$ then $f_x^{-1}[\{A\subseteq\omega:\ l\in A\}]$ is open and $f_x^{-1}[\{A\subseteq\omega:\ l\notin A\}]$ is closed. Hence, $f_x$ is of Baire class $1$, for every $x\in[0,1]$, and $f_{x}^{-1}[\I]$ is analytic, for each analytic ideal $\I$. It follows that $\conv(\I,(\alpha_n))$ is analytic for every analytic ideal $\I$.
\end{proof}

\section{The counterexample}

The following sequence of lemmas will lead us to the solution of M. Hru\v{s}\'ak's problem.

\begin{lemma}
\label{1}
The following are equivalent for every ideal $\I$ on $\omega$ and every sequence $(\alpha_n)\in\cS$:
\begin{itemize} 
\item[(a)] $\I$ can be extended to a $P^+$-ideal;
\item[(b)] $\conv(\I,(\alpha_n))$ can be extended to a $P^+$-ideal.
\end{itemize}
\end{lemma}

\begin{proof}
(b)$\implies$(a): Suppose that $\conv(\I,(\alpha_n))\subseteq\J$ for some $P^+$-ideal $\J$ (on $[0,1]\cap\mathbb{Q}$). If each convergent sequence is in $\J$, then $\conv\subseteq\J$, which contradicts the choice of $\J$ (as $\conv$ cannot be extended to any $P^+$-ideal -- see the Introduction for details). Thus, there is a convergent sequence $A\in\conv\setminus\J$ (so also $A\notin\conv(\I,(\alpha_n))$). Let $x\in[0,1]$ be the limit of $A$. Without loss of generality we may assume that $A\subseteq B(x,\alpha_0)$. Note that $\I\leq_K\conv(\I,(\alpha_n))|A\subseteq\J|A$ as witnessed by the function $f:A\to\omega$ given by: 
$$f(i)=n\ \Leftrightarrow\ i\in B(x,\alpha_n)\setminus B(x,\alpha_{n+1})$$ 
for all $i\in A$.

Observe that $\I$ is a subset of $\J'=\{C\subseteq\omega:\ f^{-1}[C]\in\J|A\}$ as for each $C\in\I$ we have $f^{-1}[C]\in\J|A$. To finish the proof, we will show that $\J'$ is a $P^+$-ideal. It is easy to check that $\J'$ is indeed an ideal, so let $(C_n)$ be a decreasing sequence of sets not belonging to $\J'$. Define $B_n=f^{-1}[C_n]\subseteq A$ for all $n\in\omega$. Then each $B_n$ does not belong to $\J|A$ (as $C_n\notin\J'$) and $(B_n)$ is decreasing. Since $\J$ is a $P^+$-ideal, so is $\J|A$. Hence, there is $X\subseteq A$, $X\notin\J|A$ such that $X\setminus B_n$ is finite for each $n$. Then $f[X]\setminus C_n\in\fin$ and $f[X]\notin\J'$ (as $f^{-1}[f[X]]\supseteq X\notin\J|A$). Thus, $\J'$ is a $P^+$-ideal. 

(a)$\implies$(b): Suppose that $\J$ is a $P^+$-ideal such that $\I\subseteq\J$. Consider the sequence $X=\{\alpha_n:\ n\in\omega\}$. Clearly, $\lim X=0$ and $X\notin\conv(\I,(\alpha_n))$. Let $f:\omega\to X$ be given by $f(n)=\alpha_n$ for all $n$. Observe that $f$ witnesses that $\I$ and $\conv(\I,(\alpha_n))|X$ are isomorphic. Consider the ideal $\J'=\{A\subseteq\mathbb{Q}\cap[0,1]:\ f^{-1}[A\cap X]\in\J\}$ (which is isomorphic with $\J\oplus\cP(\omega)$). Note that $\conv(\I,(\alpha_n))\subseteq\J'$. Moreover, similarly as above it can be shown that $\J'$ is a $P^+$-ideal. This ends the proof.
\end{proof}

\begin{lemma}
\label{2}
We have $\conv\not\leq_K\conv(\I,(\alpha_n))$, for every $(\alpha_n)\in\cS$ and every ideal $\I$ on $\omega$.
\end{lemma}

\begin{proof}
We will use the following characterization: $\conv\leq_K\J$ if and only if there is a countable family $\{X_n:\ n\in\omega\}\subseteq[\bigcup\J]^\omega$ such that for every $A\notin\J$ there is $n\in\omega$ such that both $A\cap X_n$ and $A\setminus X_n$ are infinite (cf. \cite[Theorem 2.4.3]{Meza}).

Let $\{X_n:\ n\in\omega\}\subseteq[[0,1]\cap\mathbb{Q}]^\omega$. We will recursively define $(i_n)\in 2^\omega$, $(Y_n)\subseteq[[0,1]\cap\mathbb{Q}]^\omega$ and a sequence $(I_n)$ of closed subintervals of $[0,1]$ such that:
\begin{itemize}
\item[(a)] $I_{n+1}\subseteq \text{int}(I_n)$ and the length of $I_n$ is at most $\frac{1}{2^{n+1}}$, for all $n\in\omega$;
\item[(b)] $Y_{n+1}\subseteq Y_n$ and $\overline{Y_n}=I_n$, for all $n\in\omega$;
\item[(c)] if $i_n=0$ then $Y_n\cap X_n=\emptyset$ and if $i_n=1$ then $Y_n\subseteq X_n$.
\end{itemize}
First, since $[0,1]=\overline{[0,1]\cap\mathbb{Q}}=\overline{Y_0^0}\cup\overline{Y_0^1}$, where $Y_0^0=([0,1]\cap\mathbb{Q})\setminus X_0$ and $Y_0^1=X_0$, there is $i_0\in\{0,1\}$ such that $Y_0^{i_0}$ is dense in some open subinterval of $[0,1]$. Find a closed interval $I_0\subseteq[0,1]$ of length at most $\frac{1}{2}$ such that $Y_0^{i_0}$ is dense in $I_0$ and define $Y_0=I_0\cap Y_0^{i_0}$. Note that $\overline{Y_0}=I_0$. At step $n+1$, if all $i_j$, $I_j$ and $Y_j$ for $j\leq n$ are already defined, observe that $I_n=\overline{Y_n}=\overline{Y_{n+1}^0}\cup\overline{Y_{n+1}^1}$, where $Y_{n+1}^0=Y_n\setminus X_{n+1}$ and $Y_{n+1}^1=Y_n\cap X_{n+1}$. Then there is $i_{n+1}\in\{0,1\}$ such that $Y_{n+1}^{i_{n+1}}$ is dense in some open subinterval of $I_n$. Find a closed interval $I_{n+1}\subseteq \text{int}(I_n)$ of length at most $\frac{1}{2^{n+2}}$ such that $Y_{n+1}^{i_{n+1}}$ is dense in $I_{n+1}$ and define $Y_{n+1}=I_{n+1}\cap Y_{n+1}^{i_{n+1}}$. Then $\overline{Y_{n+1}}=I_{n+1}$ and $Y_{n+1}\subseteq Y_n$.

Once the recursion is completed, let $x\in[0,1]$ be the unique point such that $\{x\}=\bigcap_n I_n$. Denote: 
$$k_0=\min\{k\in\omega:\ \text{int} (I_0)\cap(B(x,\alpha_{k})\setminus B(x,\alpha_{k+1}))\neq\emptyset\}.$$ 
Pick $x_k$, for each $k\geq k_0$, such that $x_k\in Y_{m_k}\cap(B(x,\alpha_{k})\setminus B(x,\alpha_{k+1}))$, where: 
$$m_k=\max\{n\in\omega:\ \text{int} (I_n)\cap(B(x,\alpha_{k})\setminus B(x,\alpha_{k+1}))\neq\emptyset\}$$
(item (a) guarantees that each $m_k$ is well-defined). This is possible as each $Y_n$ is dense in $I_n$ (by (b)). Note that $\lim_k m_k=\infty$ (by (a) and the choice of $x$). 

Observe that $X=\{x_k:\ k\in\omega\}\notin\conv(\I,(\alpha_n))$. Indeed, let $B_0,\ldots,B_m$ be sequences converging $\I$-quickly with respect to $(\alpha_n)$ and assume to the contrary that $X\subseteq\bigcup_{i\leq m}B_i$. Since $\I$ is an ideal, without loss of generality we may assume that $\lim B_i\neq\lim B_j$ whenever $i,j\leq m$ are distinct (as a union of finitely many sequences converging $\I$-quickly with respect to $(\alpha_n)$ to the same limit is a sequence converging $\I$-quickly with respect to $(\alpha_n)$). Since $\lim X=x$, there is $i_0\leq m$ such that $x=\lim B_{i_0}$. Then $X\cap\bigcup_{i\leq m,i\neq i_0}B_i$ is finite and $X\setminus\bigcup_{i\leq m,i\neq i_0}B_i$ does not converge $\I$-quickly with respect to $(\alpha_n)$ (as it intersects almost all $B(x,\alpha_{k})\setminus B(x,\alpha_{k+1})$). Hence, $X\setminus\bigcup_{i\leq m,i\neq i_0}B_i$ cannot be covered by $B_{i_0}$. This shows that $X\notin\conv(\I,(\alpha_n))$.

We claim that for each $n\in\omega$ either $X\cap X_n$ or $X\setminus X_n$ is finite. Indeed, fix any $n\in\omega$. If $i_n=0$ then $Y_j\cap X_n=\emptyset$ for all $j\geq n$ (by (b) and (c)). Thus, by (b) and $\lim_k m_k=\infty$ we get $X\cap X_n\subseteq X\setminus Y_n\in[\mathbb{Q}\cap[0,1]]^{<\omega}$. On the other hand, if $i_n=1$ then (b) and (c) give us $Y_j\subseteq X_n$ for all $j\geq n$. Therefore, similarly as before, $X\setminus X_n\subseteq X\setminus Y_n\in[\mathbb{Q}\cap[0,1]]^{<\omega}$. This finishes the proof.
\end{proof}

\begin{lemma}
\label{3}
The ideal $\conv(\I_d,(\frac{1}{2^{n+1}}))$ is $\bf{\Sigma^0_6}$.
\end{lemma}

\begin{proof}
Recall that: 
$$\I_d=\Exh(\phi)=\left\{A\subseteq\omega:\ \lim_n\phi(A\setminus[0,n])=0\right\},$$ 
where $\phi:\mathcal{P}(\omega)\to[0,1]$ given by: 
$$\phi(A)=\sup_{n\in\omega}\frac{|A\cap[0,n]|}{n+1},$$ 
for all $A\subseteq\omega$, is a lower semicontinuous submeasure, i.e., it satisfies $\phi(\emptyset)=0$, $\phi(A)\leq\phi(A\cup B)\leq\phi(A)+\phi(B)$ and $\phi(A)=\lim_n\phi(A\cap[0,n])$ (lower semicontinuity) for all $A,B\subseteq\omega$ (see \cite[Example 1.2.3.(d)]{Farah}). 

We will need the following observation: if $G\in[\omega]^{<\omega}$ and $\min G>0$ then $\phi((G-1)\cup G\cup(G+1))\leq 4\phi(G)$ (here $B+d=\{b+d:\ b\in B\}$ for $d\in\mathbb{Z}$ and $B\subseteq\omega$). Indeed, it is obvious that $\phi(G+1)\leq\phi(G)$. Moreover, since $G$ is finite, there is $d\in\omega$ such that $\phi(G-1)=\frac{|(G-1)\cap[0,d]|}{d+1}$. If $|(G-1)\cap[0,d]|=1$ then $\phi(G-1)=\frac{1}{d+1}\leq\frac{2}{d+2}=2\phi(G)$. If $|(G-1)\cap[0,d]|\neq 1$ then we have:
$$\phi(G-1)=\frac{|(G-1)\cap[0,d]|}{|(G-1)\cap[0,d]|-1}\cdot\frac{|(G-1)\cap[0,d]|-1}{d+1}\leq$$
$$2\cdot\frac{|(G-1)\cap[0,d]|-1}{d+1}\leq 2\phi(G).$$
Thus, $\phi((G-1)\cup G\cup(G+1))\leq \phi(G-1)+\phi(G)+\phi(G+1)\leq 4\phi(G)$.

Denote $\alpha_n=\frac{1}{2^{n+1}}$ for all $n$. We will show that $A\in\conv(\I,(\alpha_n))$ is equivalent to:
$$\exists_{k\in\omega}\ \forall_{m\in\omega}\ \exists_{n\in\omega}\ \exists_{H\in[\mathbb{Q}\cap[0,1]]^{<\omega}}\ \forall_{l\in\omega}\ \exists_{x_0,\ldots,x_k\in\mathbb{Q}\cap[0,1]}\ \exists_{G_0,\ldots,G_k\subseteq[n,l]}$$
$$\left(\left(\forall_{i\leq k}\ \phi(G_i)<\tfrac{1}{m+1}\right)\ \wedge\ \left(\forall_{i,j\leq k,i\neq j}\ B(x_i,\alpha_n)\cap B(x_j,\alpha_n)=\emptyset\right)\ \wedge\ \right.$$
$$\left. A\subseteq H\cup\bigcup_{i\leq k}\left(B(x_i,\alpha_{l+1})\cup\bigcup_{j\in G_i}\left(B(x_i,\alpha_j)\setminus B(x_i,\alpha_{j+1})\right)\right)\ \wedge \right.$$
$$\left.\left(\forall_{i\leq k}\ \forall_{j\in G_i}\ A\cap \left(B(x_i,\alpha_j)\setminus B(x_i,\alpha_{j+1})\right)\text{ is finite}\right)\right).$$
This will finish the proof as the right-hand side condition is clearly $\bf{\Sigma^0_6}$.

($\Rightarrow$): Let $A\in\conv(\I,(\alpha_n))$. Then there is $k\in\omega$ and sequences $A_0,\ldots,A_k$ converging $\I$-quickly with respect to $(\frac{1}{2^{n+1}})$ such that $A\subseteq A_0\cup\ldots\cup A_k$. Let $m\in\omega$ be arbitrary. For each $i\leq k$ there is $n_i\in\omega$ such that
$$\phi\left(\left\{j\in\omega\setminus[0,n_i):\ A_i\cap (B(\lim A_i,\alpha_j)\setminus B(\lim A_i,\alpha_{j+1}))\neq\emptyset\right\}\right)<\tfrac{1}{4(m+1)}.$$
There is also $n'\in\omega$ such that $|\lim A_i-\lim A_j|>2\alpha_{n'}$ for all $i,j\leq k$, $i\neq j$. Let $n=\max(\{n_i:\ i\leq k\}\cup\{1,n'\})$ and $H=A\setminus\bigcup_{i\leq k}B(\lim A_i,\alpha_{n+1})\in[\mathbb{Q}\cap[0,1]]^{<\omega}$. Let $l\in\omega$ be arbitrary and put $G_i=((G'_i-1)\cup G'_i\cup(G'_i+1))\cap[n,l]$ for all $i\leq k$, where
$$G'_i=\left\{j\in[n,l+1]:\ A_i\cap (B(\lim A_i,\alpha_j)\setminus B(\lim A_i,\alpha_{j+1}))\neq\emptyset\right\}.$$
Note that $\phi(G_i)<\frac{1}{m+1}$ (by the observation from the first paragraph of this proof). For each $i\leq k$ find $x_i\in[0,1]\cap\mathbb{Q}$ close to $\lim A_i$ such that:
\begin{itemize}
\item $\forall_{i,j\leq k,i\neq j}\ B(x_i,\alpha_n)\cap B(x_j,\alpha_n)=\emptyset$ (this is possible as $|\lim A_i-\lim A_j|>2\alpha_{n'}$);
\item $\forall_{i\leq k}\ B(\lim A_i,\alpha_{l+2})\subseteq B(x_i,\alpha_{l+1})$;
\item if $i\leq k$ and $n+1\leq j\leq l+1$ then $$B(\lim A_i,\alpha_j)\setminus B(\lim A_i,\alpha_{j+1})\subseteq\bigcup_{j-1\leq p\leq j+1} B(x_i,\alpha_p)\setminus B(x_i,\alpha_{p+1}).$$
\end{itemize}
Then the first two items above guarantee that $A\cap (B(x_i,\alpha_j)\setminus B(x_i,\alpha_{j+1}))$ is finite for each $i\leq k$ and $j\in G_i\subseteq[0,l]$. Moreover, we have: 
$$A\subseteq H\cup\bigcup_{i\leq k}\left(B(x_i,\alpha_{l+1})\cup\bigcup_{j\in G_i}\left(B(x_i,\alpha_j)\setminus B(x_i,\alpha_{j+1})\right)\right).$$
Hence, $A$ satisfies the right-hand side condition.

($\Leftarrow$): First, observe that if $A$ satisfies the right-hand side condition then there is $k\in\omega$ such that for each $l\in\omega$ there are $x_0,\ldots,x_k\in[0,1]\cap\mathbb{Q}$ such that $A\setminus\bigcup_{i\leq k}B(x_i,\alpha_{l+1})$ is finite. Thus, $A\in\conv$ (see \cite[Subsection 1.6]{Meza}). We need to show that each $A\in\conv\setminus\conv(\I,(\alpha_n))$ does not satisfy the right-hand side condition. 

Notice that each $A\in\conv$ is contained in finitely many convergent sequences and if $A\notin\conv(\I,(\alpha_n))$ then at least one of those sequences does not belong to $\conv(\I,(\alpha_n))$. Hence, as the right-hand side condition is closed under subsets, it suffices to show that each convergent sequence $A\notin\conv(\I,(\alpha_n))$ does not satisfy the right-hand side condition.

Suppose that $A$ is a convergent sequence not belonging to $\conv(\I,(\alpha_n))$. Let $k\in\omega$ be arbitrary and using $A\notin\conv(\I,(\alpha_n))$ find $m\in\omega$ such that for every $n\in\omega$ we have:
$$\phi\left(\left\{j\in\omega\setminus[0,n]:\ A\cap (B(\lim A,\alpha_j)\setminus B(\lim A,\alpha_{j+1}))\neq\emptyset\right\}\right)>\tfrac{4}{m+1}.$$
Fix arbitrary $n\in\omega$ and $H\in[\mathbb{Q}\cap[0,1]]^{<\omega}$. Using lower semicontinuity of $\phi$, pick $l\in\omega$ such that: 
$$\phi\left(\left\{j\in(n,l-1):\ (A\setminus H)\cap (B(\lim A,\alpha_j)\setminus B(\lim A,\alpha_{j+1}))\neq\emptyset\right\}\right)>\tfrac{4}{m+1}.$$
Observe that the above implies that $l>n+1$. 

Let $x_0,\ldots,x_k\in[0,1]\cap\mathbb{Q}$ and $G_0,\ldots,G_k\subseteq[n,l]$ be arbitrary such that:
\begin{itemize}
\item[(i)] $\phi(G_i)<\frac{1}{m+1}$ for all $i\leq k$;
\item[(ii)] $B(x_i,\alpha_n)\cap B(x_j,\alpha_n)=\emptyset$ for all $i,j\leq k$, $i\neq j$;
\item[(iii)] $A\cap B(x_i,\alpha_j)\setminus B(x_i,\alpha_{j+1})$ is finite for all $i\leq k$ and $j\in G_i$. 
\end{itemize}
Assume to the contrary that 
$$A\subseteq H\cup\bigcup_{i\leq k}\left(B(x_i,\alpha_{l+1})\cup\bigcup_{j\in G_i}\left(B(x_i,\alpha_j)\setminus B(x_i,\alpha_{j+1})\right)\right).$$ 
By (ii), (iii) and the fact that $A$ is convergent, there is $i_0\leq k$ such that $A\setminus B(x_{i_0},\alpha_{l+1})$ is finite. Then $|x_{i_0}-\lim A|\leq\alpha_{l+1}=\frac{1}{2^{l+2}}<\frac{1}{2^{n+2}}$ (by $l>n+1$) and using (ii) we get: 
$$(A\setminus H)\cap B(\lim A,\alpha_{n+1})\subseteq (A\setminus H)\cap B(x_{i_0},\alpha_{n})\subseteq $$
$$B(x_{i_0},\alpha_{l+1})\cup\bigcup_{j\in G_{i_0}}\left(B(x_{i_0},\alpha_j)\setminus B(x_{i_0},\alpha_{j+1})\right)\subseteq$$
$$B(x_{i_0},\alpha_{l})\cup\bigcup_{j\in G_{i_0}\setminus\{l\}}\left(B(x_{i_0},\alpha_j)\setminus B(x_{i_0},\alpha_{j+1})\right)\subseteq$$
$$B(\lim A,\alpha_{l-1})\cup\bigcup_{j\in G_{i_0}\setminus\{l\}}\left(B(x_{i_0},\alpha_j)\setminus B(x_{i_0},\alpha_{j+1})\right).$$

Note that if $x\in B(x_{i_0},\alpha_j)\setminus B(x_{i_0},\alpha_{j+1})$ for some $n\leq j<l$ then: 
$$x\in \bigcup_{p\in\{j-1,j,j+1\}}\left(B(\lim A,\alpha_p)\setminus B(\lim A,\alpha_{p+1})\right)=B(\lim A,\alpha_{j-1})\setminus B(\lim A,\alpha_{j+2}).$$
Indeed, $|x-x_{i_0}|<\alpha_j=\frac{1}{2^{j+1}}$ implies:
$$|x-\lim A|\leq |x-x_{i_0}|+|x_{i_0}-\lim A|<\frac{1}{2^{j+1}}+\frac{1}{2^{l+2}}\leq\frac{1}{2^{j+1}}+\frac{1}{2^{j+3}}<\frac{1}{2^{j}}=\alpha_{j-1}.$$ 
Analogously, 
$|x-x_{i_0}|\geq\alpha_{j+1}=\frac{1}{2^{j+2}}$ implies:
$$|x-\lim A|\geq |x-x_{i_0}|-|x_{i_0}-\lim A|>\frac{1}{2^{j+2}}-\frac{1}{2^{l+2}}\geq\frac{1}{2^{j+2}}-\frac{1}{2^{j+3}}=\frac{1}{2^{j+3}}=\alpha_{j+2}.$$ 

Therefore, we obtain:
$$(A\setminus H)\cap B(\lim A,\alpha_{n+1})\subseteq B(\lim A,\alpha_{l-1})\cup$$
$$\bigcup_{j\in (G_{i_0}-1)\cup G_{i_0}\cup(G_{i_0}+1)}B(\lim A,\alpha_j)\setminus B(\lim A,\alpha_{j+1})$$
However, 
$$\phi(((G_{i_0}-1)\cup G_{i_0}\cup(G_{i_0}+1))\cap(n,l-1))<4\phi(G_{i_0})<\tfrac{4}{m+1}$$ 
(by the observation from the first paragraph of this proof). This contradicts the choice of $l$ and finishes the proof.
\end{proof}

We are ready to answer the M. Hru\v{s}\'ak's question formulated in the Introduction.

\begin{theorem}
There is a $\bf{\Sigma^0_6}$ ideal $\I$ not extendable to a $P^+$-ideal and such that $\conv\not\leq_K\I$.
\end{theorem}

\begin{proof}
Consider the ideal $\conv(\I_d,(\frac{1}{2^{n+1}}))$. By Lemmas \ref{2} and \ref{3}, it is $\bf{\Sigma^0_6}$ and such that $\conv\not\leq_K\conv(\I_d,(\frac{1}{2^{n+1}}))$. Moreover, $\conv(\I_d,(\frac{1}{2^{n+1}}))$ cannot be extended to a $P^+$-ideal by Lemma \ref{1}, since $\I_d$ is not extendable to a $P^+$-ideal (cf. \cite[Theorem 4.2 and the discussion below Proposition 3.3]{Gdansk}).
\end{proof}

\section{Concluding remarks}

Proposition \ref{analytic} shows that for every Borel ideal $\I$ (and every $(\alpha_n)\in\cS$) the family $\conv(\I,(\alpha_n))$ is analytic, however in the general case we were not able to show that it is Borel. In Lemma \ref{3} we have done it only in one special case. Thus, we have produced only one counterexample for the M. Hru\v{s}\'ak's question. It should be expected that Lemma \ref{3} can be generalized for a broader class of ideals. 

It seems that our method cannot give any counterexample for the M. Hru\v{s}\'ak's question of Borel class lower than $\bf{\Sigma^0_6}$. Indeed, if we want to apply Lemma \ref{1}, we cannot take $\I$ of class lower than $\bf{\Pi^0_3}$ (as there are no $\bf{\Pi^0_2}$ ideals, all $\bf{\Sigma^0_2}$ ideals are clearly extendable to a $\bf{\Sigma^0_2}$ ideal and there are no examples of $\bf{\Sigma^0_3}$ ideals not extendable to a $\bf{\Sigma^0_2}$ ideal). Hence, there is still a chance that the original M. Hru\v{s}\'ak's  conjecture works for instance in the case of all $\bf{\Sigma^0_4}$ ideals.

Finally, it is also worth noticing that the proof of Lemma \ref{1} heavily uses the fact that $\conv$ cannot be extended to a $\bf{\Sigma^0_2}$ ideal. This suggests that $\conv$ is indeed a critical ideal for the property of extendability to a $\bf{\Sigma^0_2}$ ideal. However, the potential characterization of this property with the use of $\conv$ has to be more complicated than just the condition "$\conv\not\leq_K\I$".


\bibliographystyle{amsplain}
\bibliography{Extendability-references}

\end{document}